\documentclass{article}
\usepackage{graphicx} 

\usepackage{amsmath,amssymb,amsthm,amsfonts}
\usepackage{amsfonts}
\usepackage{pinlabel}
\usepackage{color}

\newtheorem{theorem}{Theorem}
\newtheorem{lemma}[theorem]{Lemma}

\theoremstyle{definition}

\begin{document}

\title{Efficient Detection of Borromean Linking in Ellipses}

\author{Jonathan Strange\thanks{Department of Mathematics and Statistics, California State University Long Beach, Long Beach, CA  90840, USA (Jonathan.Strange01@student.csulb.edu)} \and Ryan Blair\thanks{Department of Mathematics and Statistics, California State University Long Beach, Long Beach, CA  90840, USA (ryan.blair@csulb.edu)}\and Alexander R. Klotz\thanks{Department of Physics and Astronomy, California State University Long Beach, Long Beach, CA  90840, USA (Alex.Klotz@csulb.edu)}}


\maketitle

\begin{abstract}
    
We describe a method by which the number of intersections one ellipse makes inside the plane of another can be determined.  The method is based on applying a transformation that reverts one ellipse to the unit circle, and examining the intersection points of the transformed second ellipse with the unit circle. This may be used to efficiently determine Hopf linking between two ellipses, or Borromean linking between three ellipses.
\end{abstract}

\section{Introduction}

Topological links occur in many physical systems including astrophysical magnetic fields \cite{berger}, fluid vortices \cite{irvine}, and biological and synthetic polymers \cite{vinograd, depablo}. The key mathematical tool for determining whether two systems are linked is the Gauss linking number (GLN), a double integral over two space curves that determines the number of times one curve passes through another \cite{epple1998orbits}. This is broadly applied to understand DNA supercoiling \cite{supercoil}, the topology of kinetoplast DNA \cite{he}, the orbits of asteroids \cite{epple1998orbits}, and stellar magnetohydrodynamics \cite{berger}. However, there are forms of topological linking that the GLN does not capture. Examples include the Whitehead link between a circle and a twisted loop, and Borromean rings, which feature three unlinked circles that cannot be separated. Originally the heraldic symbol of the House of Borromeo, Borromean rings are the simplest \textit{Brunnian} link, a class of links in which any subset of the components are unlinked \cite{brunnian}.

In dense solutions or melts of polymers, entanglements between the constituent molecules dictate material properties \cite{entanglement}. Pairwise interactions are mainly considered, but higher-order multi-polymer interactions analogous to Borromean linking may influence behavior. A small number of simulation studies of polymer melts have detected such links \cite{ubertini, cao2014simulating}. An \textit{Olympic gel} is a network of topologically linked circular molecules \cite{raphael1997progressive}, and kinetoplast DNA is a natural example that occurs in trypanosome parasites \cite{shapiro1995structure}.  Several analyses of the topology of kinetoplast DNA \cite{davide2, he} used the GLN to determine the structure of the linking between DNA minicircles. These have revealed an intricate but disorganized network structure featuring a combination of Hopf and Solomon linking (GLN of $\pm$1 or 2), but these analyses were unable to detect other types of links including Borromean rings. Due to the computational difficulty of computing the necessary knot invariants, it is unknown if Borromean linking is likely in dense Olympic gel networks, or what effect this has on a gel's viscoelasticity. 

In principle there are many knot invariants that can be used to detect Borromean rings. In practice however, few of these have been successfully implemented as algorithms that can take space curves as input and yield a topological classification as output. The extension of the Gauss linking number to three curves is the Milnor triple integral and the related Massey product \cite{milnor}. Several works have attempted to show that these can detect Borromean rings \cite{akhmet2002milnor,evans1992hierarchy,mellor2003geometric}, but did not demonstrate the ability to do so with arbitrary curves. Others have developed algorithms that compute such parameters, but typically only tested these on a single configuration of known topology \cite{velavick, alokbia2022new}. The Python package Topoly \cite{topoly} is able to detect Borromean rings from Cartesian coordinates based on a computation of the Jones polynomial, as can a parallel implementation in Python by Barkataki and Panagiotou \cite{barkataki2025parallel}. Such algorithms have been used to examine the formation of Borromean linking in random polymer melts \cite{ubertini} and analyze open Borromean curves \cite{barkataki2022jones}. Because the number of possible triple links grows as the cube of the number of molecules, it is desirable to be able to efficiently check for Borromean linking.

A recent analysis by one of the authors of this manuscript \cite{klotz2024borromean} showed that Borromean links could arise in dense packings of random rectangles with mutually perpendicular normals. That work exploited the geometry of perpendicularly aligned rectangles to rapidly determine whether they formed Borromean connections, rather than relying on knot invariants. The determination was based on the idea of ``piercing'' in which two opposite sides of one rectangle pass through the plane of another. If three rectangles (\textit{A, B, C}) are arranged such that $A$ pierces $B$, $B$ pierces $C$, and $C$ pierces $A$, the rectangles form Borromean links, and this can be determined from the coordinates of their vertices. 

Can this quick geometric detection of Borromean links be extended to other shapes? Freedman and Skora \cite{FreedmanSkora1987} showed that no Brunnian link can be built out of round circles, but a Borromean link can be constructed out of ellipses, as in Figure 1b. More complex Brunnian links exist, but Howards showed that the Borromean rings are the only Brunnian link of 3 components that can be built out of convex planar curves \cite{Howards2006}. In particular, Howards' results can be used to give a classification of Borromean linking achieved by three ellipses based on the intersection of three disks bounded by the ellipses. In an effort to keep the results of this paper self-contained, we provide a proof of this classification theorem from basic principles in Section \ref{Sec:Proof}. This proof is  achieved by constructing isotopies of surfaces and ambient isotopies of curves using cut-and-paste topology techniques. Following the proof, we extend the idea of geometric piercing to develop an algorithm in Section \ref{Sec:imp} for determining whether ellipses may be linked or pierced. Our method involves taking a transformation that maps one ellipse to the unit circle, applying it to two ellipses, and identifying the points of intersection between the second ellipse and the XY plane (Fig. 1a). The algorithm based on our method can identify linking between two ellipses and Borromean linking between three ellipses approximately 150 times faster than existing algorithms based on the Gauss linking number and the Jones polynomial respectively.
\color{black}

\section{Summary of Main Result}

Ellipses can be described as unit circles transformed by a matrix. Given two ellipses transformed by distinct matrices and applying the inverse of one ellipse's transformation matrix to both, one will be returned to the unit circle and the other will not. If the second ellipse intersects the XY plane after the inverse transformation is applied, the number of intersections within the unit circle determines the number of intersections between the ellipses (Fig. 1a). If three ellipses each pass through each other twice in a certain order, they can form Borromean rings (Fig. 1b).

\begin{figure}
    \centering
    \includegraphics[width=1\linewidth]{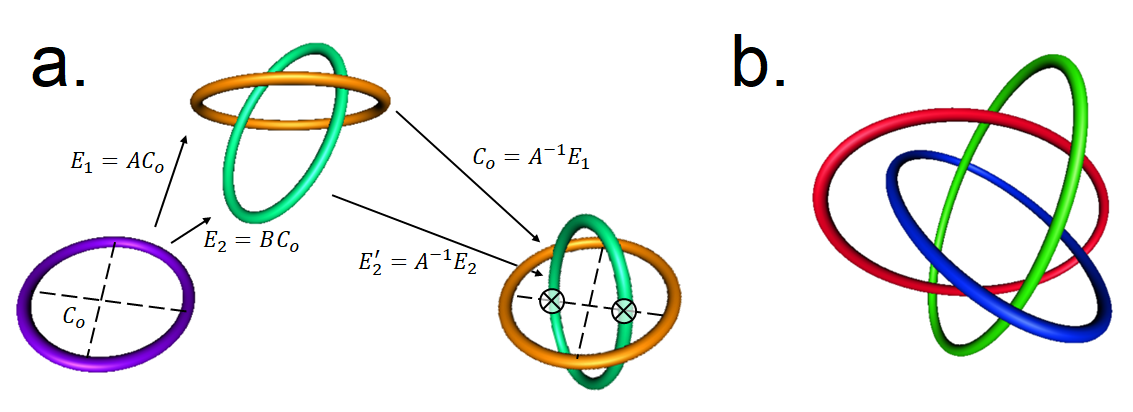}
    \caption{a. A unit circle $C_o$ is transformed by two different matrices $\textbf{A}$ and $\textbf{B}$ into two distinct ellipses $E_1$ and $E_2$.  In this example, when the inverse transformation of one ellipse is applied to both, one returns to the unit circle and the other (in this case) intersects the unit disk twice. The locations of the intersections after this transformation may be used to determine the number of piercings between the two ellipses. b. Borromean rings formed from three interpenetrating ellipses.}
    \label{fig:placeholder}
\end{figure}

Two disjoint ellipses in three-dimensional space may pass through each other zero, one, or two times, but the Gauss linking number cannot distinguish between zero and two passages. There are several ways for two ellipses to be situated: one ellipse may never intersect the plane of another (Fig. 2a), one ellipse may intersect the plane of another outside the second ellipse (Fig. 2b), one ellipse may intersect the plane of another once inside and once outside the second ellipse, forming a Hopf link (Fig. 2c), one ellipse may intersect the plane of another twice inside the second ellipse (Fig. 2d), or one ellipse may intersect the plane of another ellipse such that it passes through zero times, but the second ellipse passes through it twice (Fig. 2e). Other configurations are possible, such as the ellipses intersecting at one or more points, or one ellipse having a single intersection point with the plane of the other. Because we are concerned with the relative positions and orientations of randomly positioned and oriented ellipses, these configurations occur with probability zero and are not the primary topic of our discussion. The main result of this paper is an algorithm that can efficiently detect Hopf linking and Borromean rings in sets of random ellipses (Sec. \ref{Sec:imp}), as well as a re-derivation of a proof \cite{Howards2006} that such an algorithm can do so correctly (Sec. \ref{Sec:Proof}).

\begin{figure}
    \centering
    \includegraphics[width=1\linewidth]{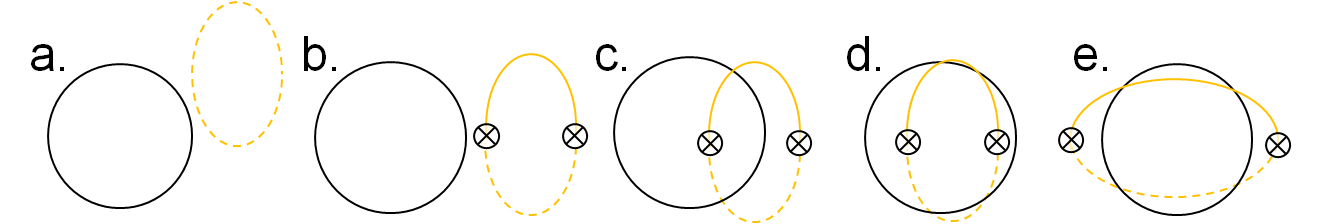}
    \caption{The five ways two ellipses may intersect each other. a. One ellipse does not intersect the plane of the other. b. One ellipse intersects the plane of the other, outside the second ellipse. c. One ellipse intersects the plane of the other, once inside and once outside the ellipse, forming a Hopf link. d. One ellipse intersects the plane of the other at two points inside the ellipse. e. One ellipse intersects the plane of the other at two points outside the ellipse, but the second ellipse passes through the first. We exclude zero-probability configurations such as the two ellipses sharing points, lying in the same plane, or one ellipse touching the other plane at one point.}
    \label{fig:placeholder}
\end{figure}
 
\color{black}

\section{Proof of Classification Theorem}\label{Sec:Proof}


\subsection{Preliminaries}

Recall that every rigid transformation of $R^3$ is a composition of rotations, translations and reflections. In this paper, a \emph{planar ellipse} is any subset of the xy-plane of the form $\{(x,y)|\frac{x^2}{a^2}+\frac{y^2}{b^2}=1\}$ where $a$ and $b$ are fixed non-zero constants. Additionally, we refer to the image of any planar ellipse under a rigid transformation of $\mathbb{R}^3$ as an \emph{ellipse}. In this section we study the properties of links in $\mathbb{R}^3$ where each component of the link is an ellipse. We call such links \emph{elliptic links}.

Whenever considering two ellipses $E_1$ and $E_2$ in $\mathbb{R}^3$, we always assume that the ellipses are in \emph{standard position} so that $E_1$ intersects the plane containing $E_2$ transversely, $E_2$ intersects the plane containing $E_1$ transversely, and $E_1\cap E_2=\emptyset$. Note that standard position can always be achieved by fixing $E_2$ and applying an arbitrarily small translation to $E_1$. Moreover, two randomly chosen ellipses will be in standard position with probability one.

Given an ellipse $E_i$, we let $P_i$ denote the plane in $\mathbb{R}^3$ that contains $E_i$ and let $D_i$ be the closed topological disk that $E_i$ bounds in $P_i$. Note that if $E_i$ and $E_j$ are in standard position, then $D_i\cap D_j$ is of one of the following forms:

\begin{enumerate}
\item $D_i\cap D_j=\emptyset$

\item $D_i\cap D_j$ is a line segment $\beta$ with one endpoint of $\beta$ in $E_i=\partial D_i$ and the other endpoint of $\beta$ in $E_j=\partial D_j$.

\item $D_i\cap D_j$ is a line segment embedded in the interior of $D_i$ and properly embedded in $D_j$.

\item $D_i\cap D_j$ is a line segment embedded in the interior of $D_j$ and properly embedded in $D_i$.
\end{enumerate}

We will refer to the above cases as IF1, IF2, IF3 and IF4 , respectively where IF stands for ``intersection form''. In each case, $D_i$ and $D_j$ intersect transversely as smooth $2$-manifolds embedded in $\mathbb{R}^3$. Note that if $D_i\cap D_j$ is IF2, then $E_i\cup E_j$ is a Hopf link. In IF3 we say that $D_j$ \emph{pierces} $D_i$ and we denote this configuration by $D_j\rightarrow D_i$. Similarly, in IF4 we say $D_i$ pierces $D_j$ and write $D_i\rightarrow D_j$.

A \emph{Borromean link} is any link ambient isotopic to the link depicted in Figure \ref{fig:ellipticBorromean}. As illustrated in the figure, it is possible to construct a three-component elliptic link that is a Borromean link. As a consequence of standard position, given an elliptic link $K$ with components $E_1$, $E_2$ and $E_3$, either $D_1\cap D_2 \cap D_3 =\emptyset$ or $D_1\cap D_2 \cap D_3 =\{p\}$ for some point $p$. In the latter case, we say $p$ is the \emph{triple point} of $K$.

\begin{figure}[htbp]
    \centering
    \includegraphics[width=2in]{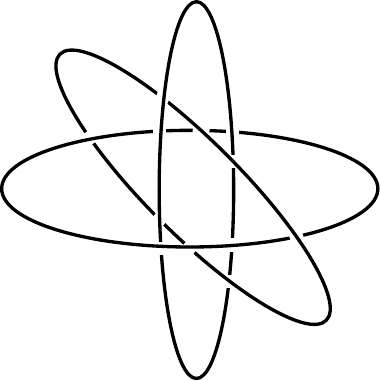}
    \caption{An elliptic Borromean link. }
    \label{fig:ellipticBorromean}
\end{figure}

In \cite{Howards2006}, Howards studied Brunnian links of 3 components
that can be built out of convex planar curves. Since all ellipses in $\mathbb{R}^3$ are convex planar curves, Howards' results can be used to give a classification of Borromean linking achieved by three ellipses. The relevant result is the following:

\begin{lemma}\label{Lem:How}(Lemma 2.2 of \cite{Howards2006})
    The Borromean rings, are the unique link with three convex planar components, bounding planar disks, with one triple point in their intersection and one exterior arc and one interior arc on each of the three disks.
\end{lemma}

Note that Lemma \ref{Lem:How} implies Theorem \ref{thm:main} when the language of exterior and interior arcs is translated into the language of piercing. However, we still prove Theorem \ref{thm:main} from basic principles in the next section in an effort to keep this paper self-contained.

Recall that every smoothly embedded $2$-sphere $S$ in $\mathbb{R}^3$ cuts $\mathbb{R}^3$ into two connected components, a $3$-ball $B^S_1$ and a punctured $3$-ball $B^S_2$. A link $K$ in $\mathbb{R}^3$ is \emph{split} if there exists a smoothly embedded $2$-sphere $S$ so that some components of $K$ are contained in $B^S_1$ and some components of $K$ are contained in $B^S_2$. Given an elliptic link $K$ with component $E_i$, let $\overline{\eta(D_i)}$ denote a small closed regular neighborhood of $D_i$ in $\mathbb{R}^3$. Then $S_i= \partial (\overline{\eta(D_i)})$ is a smoothly embedded $2$-sphere such that $D_i \subset B^{S_i}_1$. In the argument that follows, we will use the spheres $S_i$ to define isotopies and to identify split links.

\subsection{Classification of Elliptic Borromean Links}

\begin{theorem}\label{thm:main}
Given a 3-component elliptic link $K=E_1 \cup E_2 \cup E_3$ in standard position in $\mathbb{R}^3$, $K$ is a Borromean link if and only if $K$ has a triple point and (after possibly relabeling $E_1$, $E_2$ and $E_3$) $D_1\rightarrow D_2 \rightarrow D_3 \rightarrow D_1$.
\end{theorem}

\begin{proof}
Assume that $K=E_1 \cup E_2 \cup E_3$ is a 3-component elliptic link in standard position and that $K$ is a Borromean link. First, we recall some basic properties of Borromean links. Note that every two-component sublink of a Borromean link is an unlink. In particular, no two-component sublink of $K$ is a Hopf link. Hence, no $D_i\cap D_j$ is IF2. Additionally, Borromean links are not split links. Hence, there is no smoothly embedded $2$-sphere $S$ in the exterior of $K$ so that some components of $K$ are contained in $B^S_1$ and some components of $K$ are contained in $B^S_2$.

Let $\{i,j,k\}=\{1,2,3\}$ and suppose that both $D_i\cap D_j$ and $D_i\cap D_k$ are IF1. Then $S_i= \partial (\overline{\eta(D_i)})$ is a smoothly embedded $2$-sphere demonstrating that $K$ is a split link, a contradiction. Since both $D_i\cap D_j$ and $D_i\cap D_k$ are not IF1 and not IF2, then, after possibly rechoosing labels $i,j$ and $k$, we can assume that $D_i\cap D_j$ is IF4, equivalently $D_i\rightarrow D_j$.

We first suppose that $D_j \cap D_k$ is IF1, equivalently $D_j\cap D_k=\emptyset$. In this case we can construct a sphere in the exterior of $K$ demonstrating that $K$ is split. To see this, we begin by examining the line segment $\alpha = D_i\cap D_j$. Note that $\alpha$ is a properly embedded arc in $D_i$ that cuts $D_i$ into two disks, $D'_i$ and $D''_i$. Note that since $D_j\cap D_k=\emptyset$, then $K$ can not have a triple point. Since $K$ has no triple point, $D_i \cap D_k$ is disjoint from $\alpha$ and $D_i \cap D_k$ is empty or is contained in exactly one of $D'_i$ or $D''_i$. Without loss of generality, we assume $D_i \cap D_k$ is empty or is contained in $D''_i$ and $D_k\cap D'_i=\emptyset$. We perform an isotopy of $D_j$ supported in a neighborhood of $D'_i$ after which the resulting disk $D^*_j$ is disjoint from both $D_i$ and $D_k$. See Figure \ref{fig:IsoDisjoint1}. Hence, $S^*_j= \partial (\overline{\eta(D^*_j)})$ is a smoothly embedded $2$-sphere in the exterior of $K$ that separates $E_j$ from $E_i$ and $E_k$. Such a sphere implies that $K$ is a split link, a contradiction.



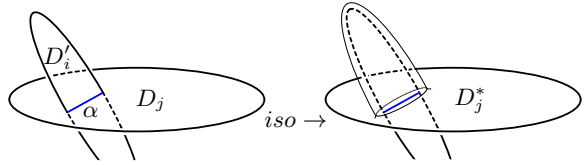
\begin{figure}[htbp]
\labellist
\small
\pinlabel{$iso \rightarrow$} at 252 65
\pinlabel{$\alpha$} at 72 72
\pinlabel{$D'_i$} at 43 120
\pinlabel{$D_j$} at 125 82
\pinlabel{$D^*_j$} at 405 82
\endlabellist
    \centering
    \includegraphics[width=3in]{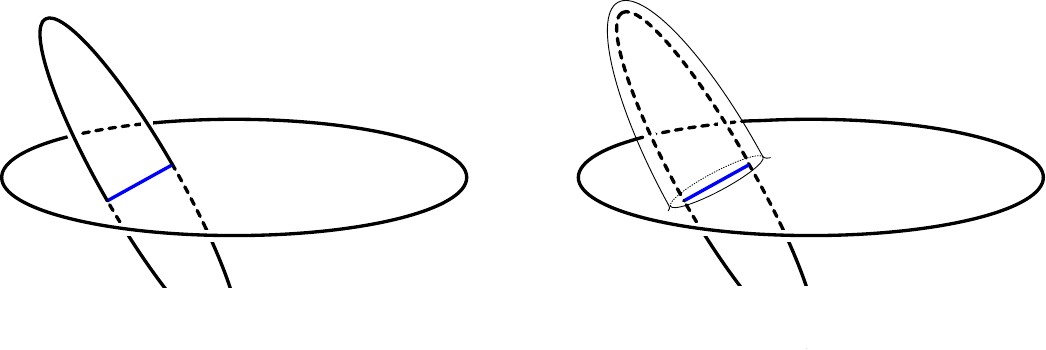}
    \caption{An isotopy of $D_j$ resulting in $D^*_j$, which is disjoint from $D_i$. }
    \label{fig:IsoDisjoint1}
\end{figure}


Next, we suppose that $D_j \cap D_k$ is IF3, equivalently $D_k\rightarrow D_j$. Define the line segments $\alpha = D_i\cap D_j$ and $\beta = D_j\cap D_k$. Note that $\alpha$ is a properly embedded arc in $D_i$ that cuts $D_i$ into two disks, $D'_i$ and $D''_i$. Similarly, $\beta$ is a properly embedded arc segment in $D_k$ cutting it into disks $D'_k$ and $D''_k$. Additionally, suppose that $\alpha \cap \beta = \emptyset$, equivalently $K$ has no triple point. In this case, if $D_i\cap D_k$ is non-empty, then $D_i\cap D_k$ is contained in exactly one of $D'_i$ and $D''_i$ and exactly one of $D'_k$ and $D''_k$. Without loss of generality, suppose $D_i\cap D_k\subset D''_i$ and $D_i\cap D_k\subset D''_k$. In particular, $D'_i$ is disjoint from $D_j$ and $D_k$ and $D'_k$ is disjoint from $D_i$ and $D_j$. We perform an isotopy of $D_j$ supported in a neighborhood of $D'_i\cup D'_k$ after which the resulting disk $D^*_j$ is disjoint from both $D_i$ and $D_k$. Hence, $S^*_j= \partial (\overline{\eta(D^*_j)})$ is a smoothly embedded $2$-sphere demonstrating that $K$ is a split link, a contradiction. See Figure \ref{fig:IsoDisjoint2}. Thus, we can assume that $\alpha \cap \beta \neq \emptyset$. In this case, $\alpha \cap \beta$ is a single point $p$, the triple point of $K$. Note that $\overline{\eta(D_i\cup D_j)}$ is a closed $3$-ball $B$, $D_j\cap B=H$ is a closed disk embedded in the interior of $D_j$, and $H$ cuts $B$ into two $3$-balls $B'$ and $B''$. We perform an isotopy of $D_j$ supported in a neighborhood of $B'$ after which the resulting disk $D^*_j$ is disjoint from $B$, $D_i$ and $D_j$. Hence, $S^*_j= \partial (\overline{\eta(D^*_j)})$ is a smoothly embedded $2$-sphere demonstrating that $K$ is a split link, a contradiction. See Figure \ref{fig:IsoDisjoint3}. Since in all configurations we arrive at give a contradiction, we can assume that $D_j \cap D_k$ is not IF3.

\begin{figure}[htbp]
\labellist
\small
\pinlabel{$iso \rightarrow$} at 252 65
\pinlabel{$\alpha$} at 72 72
\pinlabel{$\beta$} at 148 70
\pinlabel{$D'_i$} at 43 120
\pinlabel{$D'_k$} at 180 120
\pinlabel{$D_j$} at 115 82
\pinlabel{$D^*_j$} at 390 82
\endlabellist
    \centering
    \includegraphics[width=3in]{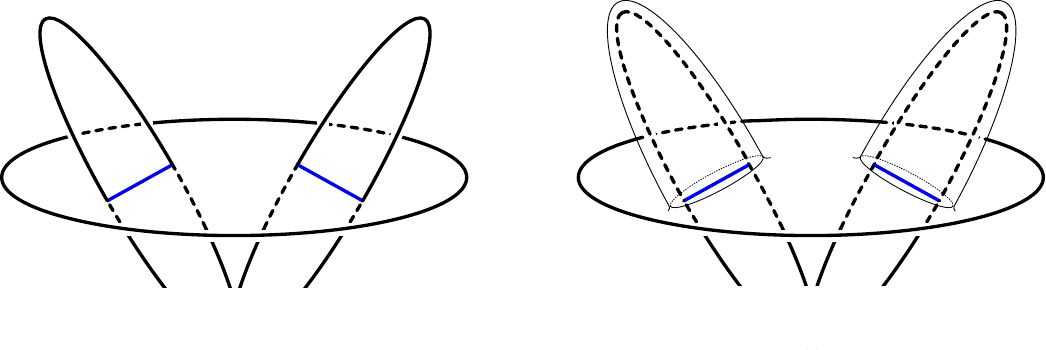}
    \caption{An isotopy of $D_j$ resulting in $D^*_j$, which is disjoint from both $D_i$ and $D_j$. }
    \label{fig:IsoDisjoint2}
\end{figure}
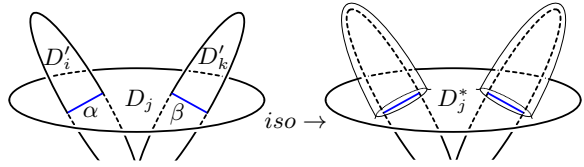

\begin{figure}[htbp]
\labellist
\small
\pinlabel{$iso \rightarrow$} at 252 65
\pinlabel{$D'_i$} at 80 122
\pinlabel{$D'_k$} at 129 123
\pinlabel{$D_j$} at 170 82
\pinlabel{$D^*_j$} at 440 82
\endlabellist
    \centering
    \includegraphics[width=3in]{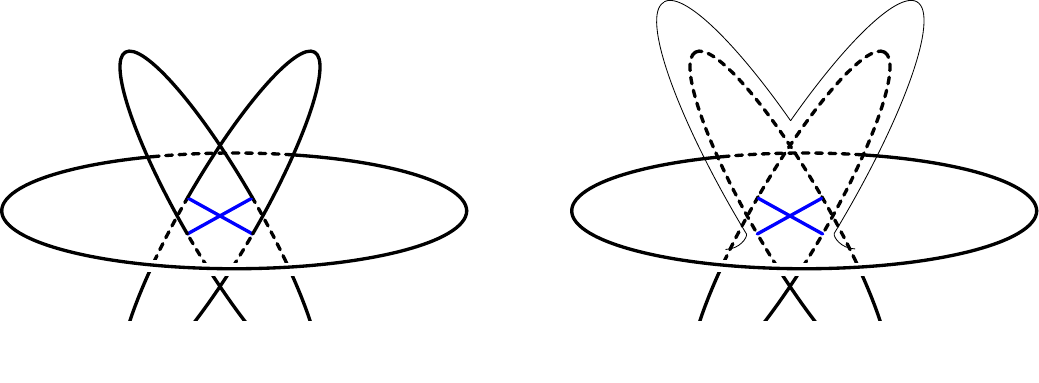}
    \caption{Another possible isotopy of $D_j$ resulting in $D^*_j$, which is disjoint from both $D_i$ and $D_j$. }
    \label{fig:IsoDisjoint3}
\end{figure}

At this point in the argument we have shown that $D_i\rightarrow D_j$ and that $D_j\rightarrow D_k$. Using the knowledge that $D_j\rightarrow D_k$ we can apply the same argument that is given above to $D_i\cap D_k$. As above, $D_i\cap D_k$ in IF1, IF2 and IF3 all lead to a contradiction. We can conclude that $D_k\rightarrow D_i$. Hence, we have established that $$D_i\rightarrow D_j\rightarrow D_k\rightarrow D_i.$$

Suppose that $K$ has no triple point. Then the line segments $\alpha = D_i\cap D_j$, $\beta = D_j\cap D_k$ and $\gamma = D_i\cap D_k$ are all pairwise disjoint. Moreover, $\alpha$ cuts $D_i$ into two disks, $D'_i$ and $D''_i$, $\beta$ cuts $D_j$ into two disks, $D'_j$ and $D''_j$, and $\gamma$ cuts $D_k$ into two disks, $D'_k$ and $D''_k$. Additionally, after possibly relabeling these six disks, $\gamma$ is contained in the interior of $D''_i$, $\alpha$ is contained in the interior of $D''_j$, and $\beta$ is contained in the interior of $D''_k$. See Figure \ref{fig:TriplePierceNoPoint}. Next, we perform an isotopy of $D_j$ supported in a neighborhood of $D'_i$ after which the resulting disk $D^*_j$ is disjoint from $D_i$. We similarly isotope $D_k$ along $D'_j$ and $D_i$ along $D'_k$ to produce disks $D^*_k$ and $D^*_i$, respectively. Note that $E_i=\partial D^*_i$, $E_j= \partial D^*_j$, $E_k= \partial D^*_k$ and all of $D^*_i$, $D^*_j$, $D^*_k$ are all pairwise disjoint. Hence, $K$ is a $3$-component unlink, a contradiction. Consequently, $K$ must have a triple point. This concludes the forward direction of the proof.

\begin{figure}[htbp]
    \centering
    \includegraphics[width=1.5in]{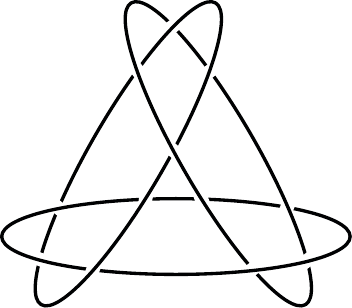}
    \caption{A 3-component elliptic link with the property that $D_i\rightarrow D_j\rightarrow D_k\rightarrow D_i$, but $D_i\cap D_j\cap D_k=\emptyset$. }
    \label{fig:TriplePierceNoPoint}
\end{figure}

We now assume that  $K=E_1 \cup E_2 \cup E_3$ is a 3-component elliptic link in standard position, that $K$ has a triple point, and that $D_1\rightarrow D_2 \rightarrow D_3 \rightarrow D_1$. Then we have line segments $\alpha = D_1\cap D_2$, $\beta = D_2\cap D_3$ and $\gamma = D_1\cap D_3$ such that $\alpha \cap \beta \cap \gamma=\{p\}$. Additionally, $\alpha$ is a properly embedded line segment in $D_1$ and $\gamma$ is embedded in the interior of $D_1$. Since $D_2 \rightarrow D_3$, then $D_2$ is disjoint from $E_3$. We can isotope $E_2$ along $D_2$ until $E_2$ is contained in the boundary of a closed  $\epsilon$-neighborhood of $\alpha$, but still disjoint from $D_1$. See the middle image in Figure \ref{fig:IsoToAlphaBeta}. After this isotopy, $D_3$ is disjoint from $E_1$ and $D_3$ is disjoint from $E_2$ outside of an $\epsilon$-neighborhood of $D_1$. Hence, we can isotope $E_3$ along $D_3$ until $E_3$ is contained in the boundary of a closed  $2\epsilon$-neighborhood of $\gamma$. Thus, $K$ has a link diagram as in rightmost image in Figure \ref{fig:IsoToAlphaBeta}. It is easy to verify that this is a diagram of the $3$-component Borromean link. 

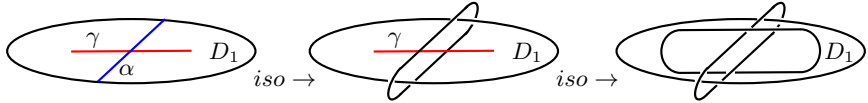
\begin{figure}[htbp]
\labellist
\small
\pinlabel{$iso \rightarrow$} at 208 17
\pinlabel{$iso \rightarrow$} at 434 17
\pinlabel{$\alpha$} at 90 23
\pinlabel{$\gamma$} at 65 45
\pinlabel{$\gamma$} at 290 45
\pinlabel{$D_1$} at 163 36
\pinlabel{$D_1$} at 389 36
\pinlabel{$D_1$} at 622 36

\endlabellist
    \centering
    \includegraphics[width=4.5in]{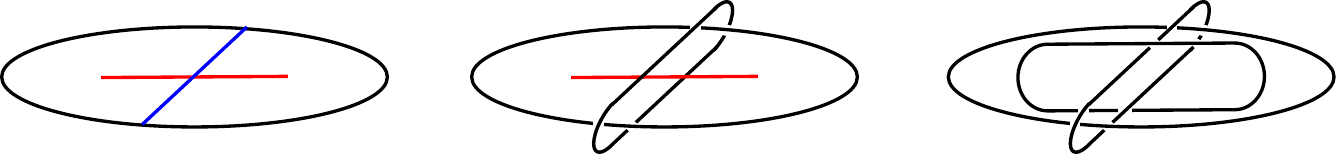}
    \caption{ If $K=E_1 \cup E_2 \cup E_3$ is a 3-component elliptic link in standard position such that $K$ has a triple point and $D_1\rightarrow D_2 \rightarrow D_3 \rightarrow D_1$, then the arcs $\alpha$ and $\beta$ guide an isotopy that shows $K$ is Borromean. }
    \label{fig:IsoToAlphaBeta}
\end{figure}

\end{proof}

\section{Implementation}\label{Sec:imp}

We represent ellipses as a transformation of the unit circle through three matrix multiplications of its coordinates plus a vector translation. The first transforms the circle by a scale factor $\rho$ and stretches it by aspect ratio $\alpha$ such that the semimajor axis is $\rho\sqrt{\alpha} $ and the semiminor axis is $\rho/\sqrt{\alpha}$:

\begin{equation}
    \textbf{M}_1=\begin{bmatrix}
\rho\sqrt{\alpha} & 0 &0\\
0 & \frac{\rho}{\sqrt{\alpha}} & 0\\
0 & 0 & 1
\end{bmatrix} 
\end{equation}

The second rotates it within the XY plane by an angle $\phi$:

\begin{equation}
    \textbf{M}_2=\begin{bmatrix}
\cos{\phi}& -\sin{\phi} &0\\
\sin{\phi} & \cos{\phi} & 0\\
0 & 0 & 1
\end{bmatrix} 
\end{equation}

The third rotates it to an arbitrary unit normal vector $\hat{n}$ by first taking the cross product $\hat{v}=\hat{z}\times\hat{n}$, \textcolor{black}{where $\hat{z}=(0,0,1)$ and $\hat{v}=(v_x,v_y,v_z)$,} and finding the sine and cosine of the angle between $\hat{n}$ and $\hat{z}$ by $s=|\hat{v}|$, \textcolor{black}{$c=\hat{n}\cdot\hat{z}$}. Then we define an intermediate matrix:

\begin{equation}
    \textbf{V}=\begin{bmatrix}
0& -v_z &v_y\\
v_z & 0 & -v_x\\
-v_y & v_x & 0
\end{bmatrix} 
\end{equation}
from which we define the final rotation matrix:
\begin{equation}
    \textbf{M}_3=\mathbb{I}_3+\textbf{V}+\frac{1-c}{s^2}\textbf{V}^2
\end{equation}
With these, we can define the transformation matrix $\textbf{M}=\textbf{M}_3\textbf{M}_2\textbf{M}_1$. Note that the third column of the transformation matrix is \textcolor{black}{a unit} normal vector of the ellipse. Alternately, the transformation matrix may be constructed by the multiplication of the scaling matrix and three Euler angle matrices. Although it is possible to incorporate three-dimensional translation as a multiplication of 4x4 matrices, we opt to simply add the translation vector $\vec{R}=(\Delta X, \Delta Y, \Delta Z)$ such that each ellipse may be written as:

\begin{equation}
    E(\theta)=\textbf{M}\begin{pmatrix}\cos\theta \\ \sin\theta\\ 0\end{pmatrix}+\begin{pmatrix}\Delta X \\ \Delta Y \\ \Delta Z\end{pmatrix}
\end{equation}
This yields, based on the indices of the transformation matrix, the following coordinates of the ellipse: 

\begin{equation}
E(\theta)=\begin{pmatrix}\textbf{M}_{1,1}\cos\theta+\textbf{M}_{1,2}\sin\theta+\Delta X \\\textbf{M}_{2,1}\cos\theta+\textbf{M}_{2,2}\sin\theta+\Delta Y \\ \textbf{M}_{3,1}\cos\theta+\textbf{M}_{3,2}\sin\theta+\Delta Z\end{pmatrix}
\end{equation}

To find the intersection of two ellipses, $E_1$ defined by matrix $\textbf{N}$ and center $\vec{R}_1$ and $E_2$ defined by matrix $\textbf{M}$ and center $\vec{R}_2$ we define a function describing the distance between any point on $E_2$ and the plane of $E_1$:

\begin{equation}
\begin{split}
F_p(\theta)=
\left(
\textbf{N}_{1,3}\textbf{M}_{1,1}+\textbf{N}_{2,3}\textbf{M}_{2,1}+\textbf{N}_{3,3}\textbf{M}_{3,1}
\right)  \cos\theta \\+ \left(
\textbf{N}_{1,3}\textbf{M}_{1,2}+\textbf{N}_{2,3}\textbf{M}_{2,2}+\textbf{N}_{3,3}\textbf{M}_{3,2}
\right)\sin\theta \\ +\left(
\textbf{N}_{1,3}(\Delta X_2 -\Delta X_1 )+\textbf{N}_{2,3}(\Delta Y_2 -\Delta Y_1 )+\textbf{N}_{3,3}(\Delta Z_2 -\Delta Z_1 )
\right) \\ \equiv \lambda\cos\theta+\mu\sin\theta+\nu
\end{split}
\label{eq:plane}
\end{equation}



To determine whether there is an intersection between $E_2$ and the plane of $E_1$, we find the maximum and minimum of $F_p(\theta)$ by finding the zeros of its derivative, which occur at \textcolor{black}{two critical angles, $\theta_c$ and $\theta_c + \pi$, such that}
\begin{equation}
    \theta_c=\tan^{-1}(\frac{\mu}{\lambda}).
\end{equation}
By substituting the two extrema into Eq. \ref{eq:plane} we can determine if both the maximum is greater than zero and the minimum is less than zero (when measured along the normal of $E_1$), indicating that one ellipse intersects the plane of the other, and two zeroes exists. If the zeroes exist, they occur at angles that satisfy the following expression:

\begin{equation}
    \tan{\theta_z}=\frac{\frac{1}{\lambda^2+\mu^2}\left({-\lambda\nu\pm\sqrt{\lambda^2\mu^2+\mu^4-\mu^2\nu^2}}\right)}{\frac{1}{\mu}\left(-\nu+\frac{\lambda^2\nu}{\lambda^2+\mu^2}\mp\frac{\lambda\sqrt{-\mu^2(\nu^2-\lambda^2-\mu^2)}}{\lambda^2+\mu^2}\right)}
\end{equation}
The solutions to this equation are real if the zeros exist. If the second ellipse passes twice through the plane of the first, it may do so twice outside the ellipse, once inside the ellipse, or twice inside the ellipse. To determine the number of times that the second ellipse intersects the first, we first subtract $\vec{R}_1$ and then multiply both ellipses by the inverse matrix of the first transformation matrix, $\textbf{N}^{-1}$. The composition of these transformations takes $E_1$ into a unit circle in the XY plane centered at the origin and takes $E_2$ to an ellipse centered at $\textbf{N}^{-1}\vec{R}_2-\textbf{N}^{-1}\vec{R}_1$ where $(X'_1,Y'_1,Z'_1)= \textbf{N}^{-1}\vec{R}_1$. After the appropriate transformations, the transformed ellipse $E'_2$ intersects the XY plane at the following X and Y coordinates:


\begin{equation*}
\begin{split}
    {E}'_{2X}=\left(\textbf{N}^{-1}_{1,1}\textbf{M}_{1,1}+\textbf{N}^{-1}_{1,2}\textbf{M}_{2,1}+\textbf{N}^{-1}_{1,3}\textbf{M}_{3,1}\right)\cos{\theta_z}\\+\left(\textbf{N}^{-1}_{1,1}\textbf{M}_{1,2}+\textbf{N}^{-1}_{1,2}\textbf{M}_{2,2}+\textbf{N}^{-1}_{1,3}\textbf{M}_{3,2}\right)\sin{\theta_z}\\+\left(\textbf{N}^{-1}_{1,1}\Delta X_2+\textbf{N}^{-1}_{1,2}\Delta Y_2 + \textbf{N}^{-1}_{1,3}\Delta Z_2-X'_1 \right)
    \end{split}
\end{equation*}

\begin{equation}
\begin{split}
    {E}'_{2Y}=\left(\textbf{N}^{-1}_{2,1}\textbf{M}_{1,1}+\textbf{N}^{-1}_{2,2}\textbf{M}_{2,1}+\textbf{N}^{-1}_{2,3}\textbf{M}_{3,1}\right)\cos{\theta_z}\\+\left(\textbf{N}^{-1}_{2,1}\textbf{M}_{1,2}+\textbf{N}^{-1}_{2,2}\textbf{M}_{2,2}+ \textbf{N}^{-1}_{2,3}\textbf{M}_{3,2}\right)\sin{\theta_z}\\+\left(\textbf{N}^{-1}_{2,1}\Delta X_2+\textbf{N}^{-1}_{2,2}\Delta Y_2 + \textbf{N}^{-1}_{2,3}\Delta Z_2-Y'_1 \right)
    \end{split}
\end{equation}

For each of the \textcolor{black}{(at most two) values for $\theta_z$} we find the radius of the intersection $R_z=(E'_{2X})^2+(E'_{2Y})^2$ and determine which are less than 1, be it zero, one, or two times. This yields the number of times $E_2$ passes through \textcolor{black}{the planar region bounded by $E_1$}. If zero intersections are detected, it is still possible that the first ellipse pierces the second, rather than the second piercing the first. In principle one could compute the radial minimum of the line connecting the two intersection points, but in practice it is simpler to simply swap the order of the ellipses and check again. 

To check the triple point planar intersection criterion, we define a matrix $\textbf{A}$ based on the normal vectors of each ellipse ($\vec{n}_1,\vec{n}_2,\vec{n}_3$) and a vector $\vec{b}$ based on the location of the centers of the ellipses along those vectors:

\begin{equation}
        \textbf{A}=\begin{bmatrix}
n_{1x}& n_{1y} &n_{1z}\\
n_{2x}& n_{2y} &n_{2z}\\
n_{3x}& n_{3y} &n_{3z}
\end{bmatrix} 
\end{equation}

\begin{equation}
    \vec{b}=\begin{pmatrix}\vec{n_1}\cdot\vec{ R_1}\\ \vec{n_2}\cdot\vec{R_2}\\ \vec{n_3}\cdot\vec{R_3}\end{pmatrix}
\end{equation}

The intersection point $\vec{P}$ of the three planes, if it exists is located at:
\begin{equation}
    \vec{P}=\textbf{A}^{-1}\vec{b}
\end{equation}
We then apply the inverse transformation of each ellipse to the intersection point, and determine whether the point lies within the unit circle after each transformation. If so, the planar intersection criteria is met.

If ellipses are presented as Cartesian coordinates rather than transformation matrices, the matrix can be determined by computing the gyration tensor of the ellipse's coordinates. The ratio of the two largest eigenvalues of the tensor will be the square of the aspect ratio for computing \textbf{$M_1$}, and the eigenvector with the smallest eigenvalue will be the normal vector, from which \textbf{$M_3$} may be computed. Multiplying the ellipse by \textbf{$M_3^{-1}$} will transform it to the XY plane, and the 2D gyration tensor of this transformed ellipse will have a major eigenvector that will give the angle from which \textbf{$M_2$} can be computed. 
\section{Validation and Performance}

Scripts were written in MATLAB and Python to implement the algorithms described above, and are included as ancillary files. An example of the topology determination of 100 triplets of ellipses is shown in Figure \ref{fig:mosaic}.

Three congruent, concentric ellipses with random orientation have a 1/4 probability of forming Borromean rings. This can be seen by considering that any two such ellipses $A$ and $B$ must either have $A$ piercing $B$ or vice versa, and all three pairs $AB, AC, BC$ are equally likely to be in each piercing configuration. Of the 8 possible piercing configurations, two of them have the correct sequence of piercings for Borromean linking. We generated 10 populations of 10,000 triplets of ellipses and observed that 0.251$\pm$0.002 formed Borromean links, which we believe indicates that the algorithm can correctly identify Borromean linking in ellipses.

\begin{figure}
    \centering
    \includegraphics[width=1\linewidth]{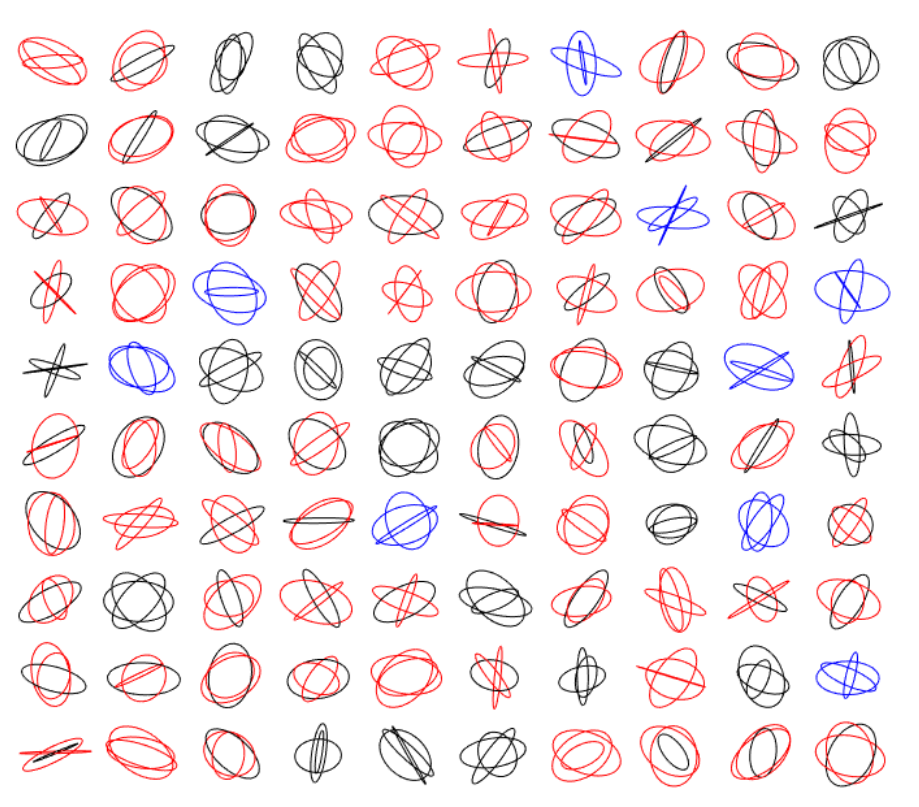}
    \caption{Mosaic of 100 sets of three ellipses projected onto the XY plane, each with a random orientation in three dimensions, an area of $\pi$, an aspect ratio of 1.618, and a center chosen randomly with a uniform distribution between $\pm$ 0.2. Black ellipses are unlinked, pairs or triplets of red ellipses share a Hopf link, and triplets of blue ellipses form Borromean links.}
    \label{fig:mosaic}
\end{figure}

As further validation of our method, we generated sets of 1000 triplets of congruent, concentric ellipses and compared the topology as inferred by our method to that determined by the Jones polynomial. For the latter, we discretized each ellipse into $N$ equi-angular points and used the Topoly package in Python \cite{topoly} to evaluate their Jones polynomial. For $N$ of 10 and above there was 100\% agreement between our method and the Jones polynomial. For $N$ of 8, discretization errors became probable and the inferred topology only matched roughly in 94\% of cases. A discretized version of our method performed less favorably, with 88\% accuracy at 10 vertices and 98\% accuracy at 100 vertices. However, the exact method is typically faster than the discrete. Our method also finds consistent results with the Jones polynomial determination when the ellipses are not concentric, such that Hopf linking is possible.

Our algorithm in MATLAB can evaluate the topology of 1000 ellipse triplets in approximately 0.2 seconds on a laptop, with the Python version taking approximately 0.5 seconds. The typical computation time for evaluating 1000 Jones polynomials of three 10-vertex ellipses in Topoly on a laptop was 32 seconds, indicating our algorithm provides roughly a factor-of-150 improvement. Another Python-based Jones polynomial computation algorithm \cite{barkataki2025parallel} produced results consistent with our ellipse piercing algorithm, with a comparable computation time to Topoly. The performance of these external algorithms (which can evaluate more complex polygonal curves than our algorithm) is hardware- and implementation-dependent, and may be improvable. 

Our method also serves as an alternative to the Gauss linking number for pairs of ellipses, for which the (absolute) value can be only 0 or 1. Comparing the linking number from our method to those calculated from discrete ellipses using the solid angle method \cite{klenin}, with 10 vertices there is typically 97\% agreement and with 20 vertices, 99\%. In the latter case, the discrete linking integral was typically a factor of 10-15 slower, but this increases quadratically with vertex count. Topoly’s GLN calculator was approximately 150 times slower than the Python version of our algorithm. Because our ellipses are not oriented (in the sense of oriented links), our algorithm in its present form cannot determine the sign of the GLN.

\section{Discussion}

With the ability to rapidly detect Hopf and Borromean linking in ellipses, we hope to extend the analysis of random shape packings to more general systems than have previously been explored. This would involve setting up systems analogous to Olympic gels and examining the likelihood of various forms of linking as a function of the density and shape distribution of the ellipses. Beyond its use in research problems, shape intersection algorithms such as the one developed here are of use in computer graphics, for example when checking for overlaps between distinct shapes. Recently, for example, an improved algorithm for Gauss linking number calculation was developed to aid in the animation of chainmail garments \cite{qujames}. 

Qualitatively, our methods for detecting pairwise linking are similar to those used by Polson et al. to detect linking between circles in simulated thermalized chainmail networks \cite{polson}. Their method relied on computing the minimum distances between circles along appropriate vectors. An extension of their linking detection algorithm was developed to search for overlaps between two solid tori. This was based on a computer graphics algorithm developed by Vranek which avoids explicit solutions of 8th-order polynomials solution by exploiting symmetries of their roots \cite{vranek}. Extending our method to solid ellipses does not immediately follow, but a lower bound on the acceptable minor radius of two tori could be derived knowing their location of intersection and angle of inclination. This would allow more realistic simulation of Olympic gel systems, which are composed of finitely-thick polymers with excluded volume interactions.

Obviously, shapes besides ellipses can become linked. It is our long-term goal to be able to detect linking between arbitrary space curves or polygons, such as simulated ring polymers, in a way that is faster than computing the Jones polynomial or Milnor triple product. In this work we have only considered ellipses, but the method of inverting transformations and searching for planar intersections may be extended to arbitrary convex planar polygons. Such a method would involve checking the intersection points of one polygon with the plane of the other, relative to the line extended from each segment in the other polygon. This could potentially be used to extend geometric percolation analysis to arbitrary regular polygons, and most generally to the convex hulls of random walks, point clouds, or similar stochastic shapes.

\subsection*{Acknowledgments: } All three authors are partially supported by NSF PREM grant DMR-2425133. 

\bibliographystyle{unsrt}
\bibliography{ellipserefs}

\end{document}